\documentclass[reqno,11pt]{amsart}
\usepackage[utf8]{inputenc}
\usepackage{graphicx, enumerate,amssymb,enumitem}
\usepackage{epsfig,amscd,amsmath,amsthm,amstext,amsbsy,upref,amsfonts,enumerate,hyperref,braket}
\usepackage{comment,pdfpages,mathrsfs}
\usepackage{xcolor}

\numberwithin{equation}{section}
\allowdisplaybreaks

\newcommand{\ol}{\overline}
\newcommand{\wt}{\widetilde}

\newcommand{\norm}[1]{\left\Vert#1\right\Vert}

\newcommand{\abs}[1]{\left\vert#1\right\vert}
\newcommand{\rl}{{\mathbb{R}}}
\newcommand{\cx}{{\mathbb{C}}}

\newcommand{\D}{\mathbb{D}}

\newtheorem{thm}{Theorem}[section]
\newtheorem*{thm*}{Theorem}
\newtheorem{prop}{Proposition}[section]
\newtheorem{lem}[prop]{Lemma}

\theoremstyle{definition}
\newtheorem{defn}[thm]{Definition}

\title[Bergman Kernels]{ Bergman Kernels of elementary Reinhardt Domains}
\author{Debraj Chakrabarti}
\email{chakr2d@cmich.edu}
\urladdr{http://people.cst.cmich.edu/chakr2d/}
\author{Austin Konkel}
\email{konke1am@cmich.edu }
\author{Meera Mainkar}
\email{maink1m@cmich.edu}
\urladdr{http://people.cst.cmich.edu/maink1m/}
\author{Evan Miller}
\email{mille7em@cmich.edu }
\address{Department of Mathematics, Central Michigan University, 
Mt. Pleasant, MI 48859, USA}
\thanks{Debraj Chakrabarti, Austin Konkel and Evan Miller were partially supported by NSF grant DMS-1600371. Austin Konkel was also supported by a Student Research and Creative Endeavors grant from Central Michigan University.}
\begin{document}
\begin{abstract}
    We study the Bergman kernel of certain domains
    in $\mathbb{C}^n$, called elementary Reinhardt domains, generalizing the classical Hartogs triangle. For some elementary Reinhardt domains, we explicitly compute the kernel,
    which is a rational function of the coordinates. For some other such domains, we show that the kernel is not a rational function. For a general elementary Reinhardt 
    domain, we obtain a representation of the 
    kernel as an infinite series.
\end{abstract}
\maketitle
\section{Introduction}

\subsection{Elementary  Reinhardt domains} Let $\D^n=\{z\in\cx^n\mid \abs{z_j}<1 \text{ for } 1\leq j \leq n\}$
denote the unit polydisc in $\cx^n,n\geq 2$, and  let 
$k=(k_1,\dots, k_n)\in \mathbb{Z}^n$ be a 
multi-index. The goal of this paper is the study of the Bergman kernel of the domain
\begin{equation}\label{eq-erd}
  \mathscr{H}(k)= \left\{z\in \D^n\mid z^k \text{ is defined, and }  \abs{z^k}<1\right\}, \end{equation}
where we use the standard multi-index convention  
$z^k=z_1^{k_1}z_2^{k_2}\dots z_n^{k_n}$,  and the only way this can fail to be defined is if its evaluation involves division by zero.
We will call the domain  $\mathscr{H}(k)$ the \emph{elementary  Reinhardt domain} associated to the multi-index $k$ (cf. \cite[pp.33 ff.]{jarpflubook}, where this terminology is used, with a slightly different definition).  A famous example of such a domain is the \emph{Hartogs triangle} 
\[ \mathscr{H}(1,-1) =\{\abs{z_1}< \abs{z_2}<1 \}\subset \cx^2,\] a well-known source of counterexamples in several complex variables (see, e.g., \cite{shawhartogs}).

It is easy to see that $\mathscr{H}(k)$ is logarithmically convex, and therefore pseudoconvex  (see \cite{range}). If the multi-index $k$ contains both positive and negative entries, then   $\mathscr{H}(k)$ is
 a Reinhardt domain with the origin as a boundary point, so it follows (see \cite{chaksibony}) that each
  holomorphic function smooth up to the boundary on $\mathscr{H}(k)$ extends to a larger, fixed
   domain, a property which is classical in the special case of the Hartogs triangle (see \cite{sibony1975,behnke1933}). Therefore, $\mathscr{H}(k)$ does not have 
   a basis of Stein neighborhoods, and is not a so-called $\mathcal{H}^\infty$-domain of holomorphy. This makes 
   domains such as $\mathscr{H}(k)$ particularly interesting from the point of view of function theory on non-smooth domains, since each smoothly bounded pseudoconvex domain is in fact an $\mathcal{H}^\infty$-domain of holomorphy (\cite{catlin,hakimsibony}).

   Recently, the unusual $L^p$-mapping properties of the Bergman projection on the \emph{generalized Hartogs triangle}
   $\mathscr{H}(m, -n)\subset \cx^2$  (where $m,n$ are coprime positive integers)  have received the attention of several authors (see \cite{chakzeytuncu,Edh16,EdhMcN16,EdhMcN16b,ChEdMc}).  In many of these investigations, 
the explicit form of the Bergman kernel of  $\mathscr{H}(m, -n)\subset \cx^2$  plays a crucial role. 
The elementary Reinhardt domains are a natural class generalizing the Hartogs triangle. Motivated by this, in this paper, we make a preliminary 
study of the Bergman kernels of  the domains $\mathscr{H}(k)$. In particular, we investigate
whether such a Bergman kernel is a rational function of the coordinates, as it indeed is if $n=2$ (see \cite{Edh16,EdhMcN16b}). Other recent attempts at  higher dimensional generalizations may be found in  \cite{park,chen, huo,chenkrantz}. For planar domains, the rationality (or algebraicity) of the Bergman kernel has important function-theoretic repercussions (see \cite{bellquadrature}). It would be interesting to see whether something similar is true for the elementary Reinhardt domains.

%%%%%%%%%%%%%%%%%%%%%%%%%%%%%%%
 From now on we will assume that the 
multi-index $k=(k_1,\dots, k_n)$ defining the domain \eqref{eq-erd} has the following properties:
\begin{enumerate}
    \item At least one of the components of the multi-index is positive, at least one of the components is negative, and no component is zero.
    
    We will call the number of positive components of $k$, the \emph{signature} $s$ of the elementary Reinhardt domain $\mathscr{H}(k)$. 
\item If $\mathscr{H}(k)$  has signature $s$, after renaming the coordinates, we will assume without loss of generality
that $k_j>0$ for $1\leq j\leq s$ and $k_j<0$ if $s+1\leq j \leq n$.
\item We will also assume without loss of generality that the numbers $k_1,\dots, k_n$ are relatively prime.
\end{enumerate}
\subsection{Explicit Formula}
For elementary Reinhardt domains of signature 1, we now give an  explicit formula for the Bergman kernel as a rational function of the coordinates.

To state the result, introduce the following notation. For integers $\lambda$ and $\mu$, let
\begin{equation}\label{eq-dcombinations}
       \mathsf{D}_\lambda(\mu) = 
       \begin{cases} 
      0 & \mu \leq -1  \text{ or } \mu \geq 2\lambda-1\\
      \mu + 1 & 0 \leq \mu \leq \lambda-1 \\
      2\lambda-1-\mu & \lambda \leq \mu \leq 2\lambda-2. \\
   \end{cases}
   \end{equation}
   It will be seen in Section~\ref{sec-explicit} below, that the seemingly complicated expression  $ \mathsf{D}_\lambda(\mu)$ arises as the number of 
   solutions in pairs of integers $(x,y)$ of the equation $x+y=\mu$ subject to the constraints $0\leq x,y \leq \lambda-1$ (see \eqref{eq-xplusy}, \eqref{eq-xbounds}, \eqref{eq-ybounds}).
\begin{thm}
\label{thm-computation}
Let $n\geq 2$, let $k_1,\dots, k_n$ be relatively prime
 positive integers,
and let \[k=(k_1, -k_2,\dots, -k_n)\in \mathbb{Z}^n\] be a multi-index.
The Bergman kernel of the elementary Reinhardt domain 
$\mathscr{H}(k)$ 
is given by:
\begin{equation}\label{eq-kscrh}
 \mathbb{B}_{\mathscr{H}(k)}(z,w)= \frac{1}{\pi^n L}
\cdot \frac{ \displaystyle{\sum_{\beta\in \mathfrak{G}} C(\beta) t^\beta}}{\displaystyle{\left(\prod_{b=2}^n t_b^{k_b} - t_1^{k_1} \right)^2 \cdot  \prod_{b=2}^n (1-t_b)^2}},   
\end{equation}
where $t=(t_1,\dots, t_n)\in\cx^n$ with $t_a=z_a\ol{w_a}$ for $1\leq a \leq n$,  and
 \begin{equation}\label{eq-cbeta}
  C(\beta) =\mathsf{D}_K\left(2K-\ell_1(\beta_1+1)-1\right)\cdot \prod_{b=2}^n \mathsf{D}_{\ell_b}(\ell_b(\beta_b+1)+\ell_1(\beta_1+1)-2K-1),   
 \end{equation}
 where the function $\mathsf{D}_*(\cdot)$ is defined in \eqref{eq-dcombinations} above, with
 \[ K = \mathrm{lcm}(k_1,\dots, k_n), \quad \ell_a = \frac{K}{k_a} \quad \text{for  } 1\leq a \leq n, \text{ and } \quad L=\prod_{a=1}^n \ell_a,
 \]
 and where  the finite collection of multi-indices $\mathfrak{G}\subset \mathbb{Z}^n  $ is defined by

    \begin{equation}
    \label{eq-G}\mathfrak{G}= \left\{(\beta_1,\dots, \beta_n)\in \mathbb{Z}^n \mid 0\leq \beta_1\leq 2k_1-2, \text{ and }
     0\leq \beta_b \leq 2k_b \text{ for each } 2\leq b\leq n \right\}.\end{equation}
 
\end{thm}

While the expression \eqref{eq-kscrh} is somewhat complicated, it generalizes and extends several 
known results in the literature.
In the special case of the classical Hartogs triangle $\mathscr{H}(1,-1)$, an explicit expression for the Bergman kernel is already found in \cite{brem}. Recently, in \cite{Edh16}, Edholm          
computed  using Bell's formula (see \eqref{eq-bell} below)  the Bergman kernels of $\mathscr{H}(1,-k)$ and
$\mathscr{H}(k,-1)$, where $k\geq 2$ is an integer (``Fat and thin generalized Hartogs triangles"), a  computational \emph{tour de force} which inspired Theorem~\ref{thm-computation}. In \cite{EdhMcN16b}, Edholm and 
McNeal studied $\mathscr{H}(m,-n)\subset\cx^2$, where $m,n$ are 
coprime positive integers, and expressed its Bergman kernel as the sum of $m$ ``sub-Bergman kernels." The sub-Bergman kernels are obtained by summing subseries of the power series \eqref{eq-bergmanseries} representing the Bergman kernel of a Reinhardt domain. These subseries consist of terms 
with monomials whose exponents are represented by straight lines 
of different slopes in the lattice point diagram of monomials, resulting in  a decomposition of the kernel into convenient pieces, which permits the explicit summation of each of the sub-kernels in closed form as a rational function, and determination of the $L^p$-regularity of each piece.  However, our formula \eqref{eq-kscrh}  shows that splitting the kernel into the sub-kernels is  unnecessary, and the main $L^p$ estimate of \cite{EdhMcN16b} could proceed directly from \eqref{eq-kscrh}.
Starting from \eqref{eq-kscrh}, we recapture below in Section \ref{sec-luke} 
the special cases considered in \cite{Edh16}. 
Theorem~\ref{thm-computation} also opens the way to generalize the 
interesting recent results related to $L^p$-regularity of the Bergman projection, duality of Bergman spaces etc. (cf. \cite{EdhMcN16b,chakzeytuncu,ChEdMc}) to higher dimensions.

\subsection{Signatures greater than 1} In signatures $s\geq 2$ (so that the ambient dimension $n\geq 3$) the situation is much less clear, and is worth further study. Here we collect a few observations
which seems to indicate that there are some fundamental differences 
between the cases $s=1$ and $s\geq 2$. In particular, it seems plausible that the Bergman kernels of  
elementary Reinhardt domains of signature $s\geq 2$  can not be represented using a simple
rational function such as \eqref{eq-kscrh}. 

Let $n\geq 2$, and 
let $1\leq s \leq n-1$. We denote by $\Omega_{n,s}$ the elementary 
Reinhardt domain of signature $s$ in $\cx^n$, where each component 
of the defining multi-index is $\pm 1$, i.e.
\begin{equation}\label{eq-Omegans}
\Omega_{n,s} = \mathscr{H}(\underbrace{1,\dots, 1}_s, \underbrace{-1,\dots, -1}_{n-s}), 
\end{equation} 
so that $\Omega_{n,s}=\{z\in \mathbb{D}^n\mid \abs{z_1\dots z_s} <\abs{z_{s+1}z_{s+2}\dots z_n}\}. $
We will call $\Omega_{n,s}$ the \emph{model elementary domain} of signature $s$. It is shown in Proposition~\ref{prop-propermap} below that the model elementary domains are branched covers of all elementary Reinhardt domains.

In Theorem~\ref{thm-monomial} below, we give an account of the coefficients of the power series expansion of $\mathbb{B}_{\Omega_{n,s}}$ by computing the $L^2$-norms of monomials $e_\alpha(z)=z_1^{\alpha_1}\dots z_n^{\alpha_n}$. This shows that the coefficient of
\[ (z_1\ol{w_1})^{\alpha_1}\dots (z_n\ol{w_n})^{\alpha_n}\]
in the power series expansion of $\mathbb{B}_{\Omega_{n,s}}$ is a polynomial 
in $\alpha_1,\dots, \alpha_n$ only if $s=1$. For $2\leq s \leq n-1$,
the coefficient is a \emph{rational} function of $\alpha_1,\dots, \alpha_n$.
From this we are able to deduce the following:
\begin{thm}\label{thm-notrational}If $n\geq 3$, then 
$\mathbb{B}_{\Omega_{n,n-1}}$ is not a rational function. 
\end{thm}
It seems highly plausible that in fact $\mathbb{B}_{\Omega_{n,s}}$ is 
a transcendental function of the coordinates unless $s=1$, though at present we are not in possession of a complete proof. If this conjecture is correct, using 
the proper map from a model domain to an arbitrary elementary Reinhardt domain, it will follow that the Bergman kernel of an elementary Reinhardt domain of signature $s\geq 2$ is transcendental.

Some further properties of the series representation of the Bergman kernel are explored in section \ref{sec-bergmanremarks} below. 

\subsection{Acknowledgements} We gratefully acknowledge the helpful
comments of Doron Zielberger, Jeff McNeal and Luke Edholm. We also thank the anonymous referee for many excellent  
suggestions which led to significant improvements. 
\section{Preliminaries}
\subsection{Bergman theory} We briefly recall some
basic facts about Bergman spaces and kernels and clarify our notation.
An extensive modern exposition of this topic  from the complex analysis point of
view is \cite{krantzbergman}, and from the operator theory 
point of view is \cite{durenbergman}.

Let $\Omega\subset\cx^n$ be a domain, i.e. a connected open set. 
Then $A^2(\Omega)$, the ($L^2$)-Bergman space of $\Omega$, is the Hilbert space of holomorphic functions which are square integrable with respect to the Lebesgue measure $dV$.
This is a so-called \emph{reproducing kernel Hilbert space},
and its reproducing kernel is the \emph{Bergman kernel}, a function
$\mathbb{B}_\Omega:\Omega\times\Omega\to\cx$, holomorphic in the first and anti-holomorphic in the second input such that for 
each $f\in A^2(\Omega)$ we have for each $z\in \Omega$ the \emph{reproducing property}:
\[f(z)= \int_\Omega f(w) \mathbb{B}_\Omega(z,w)dV(w).\]

A domain  $\Omega\subset\cx^n$ is  \emph{Reinhardt} if whenever 
$z\in \Omega$, and $\lambda\in \mathbb{T}^n$, where $\mathbb{T}^n=\{\lambda\in \cx^n\mid \abs{\lambda_j}=1 \text{ for each } 1\leq j \leq n\}$ is the unit torus, we have $(\lambda_1z_1,\dots, \lambda_nz_n)\in \Omega. $ For a Reinhardt domain, there is a canonical series representation of the Bergman kernel. For each
multi-index $\alpha\in \mathbb{Z}^n$, let $e_\alpha$ denote 
the monomial 
\begin{equation}\label{eq-ealpha}
e_\alpha(z)=z^\alpha=z_1^{\alpha_1}\dots z_n^{\alpha_n}.    
\end{equation}
Then the Bergman kernel of 
$\Omega$ has the series representation converging uniformly on 
compact subsets of $\Omega\times \Omega$:
\begin{equation}
    \label{eq-bergmanseries}\mathbb{B}_\Omega(z,w)= \sum_{\alpha\in \mathbb{Z}^n} \frac{1}{\norm{e_\alpha}^2}z^\alpha \ol{w^\alpha},
\end{equation}
where 
\begin{equation}
    \label{eq-ealphanorm}\displaystyle{\norm{e_\alpha}^2=\int_\Omega \abs{e_\alpha(z)}^2 dV(z),}
\end{equation}
and if for an $\alpha\in \mathbb{Z}^n$ the integral
\eqref{eq-ealphanorm}
diverges, the coefficient $\dfrac{1}{\norm{e_\alpha}^2}$ in \eqref{eq-bergmanseries} is taken to be zero.  An immediate consequence of this series representation is the following simple observation: if $\widetilde{\Omega}\subset\cx^n$ is the 
domain $ \wt{\Omega}= \{(z_1\ol{w_1},\dots, z_n\ol{w_n})\mid z,w\in \Omega\},$
then there is a holomorphic function $\wt{B}$ on $\wt{\Omega}$ such that 
$ \mathbb{B}_\Omega(z,w) =\wt{B}(z_1\ol{w_1},\dots, z_n\ol{w_n}),$
where  for $t\in \wt{\Omega}$,
\begin{equation}
    \label{eq-btilde}
    \wt{B}(t)= \sum_{\alpha\in \mathbb{Z}^n} \frac{1}{\norm{e_\alpha}^2}t^\alpha.
\end{equation}
Therefore, the Bergman kernel of a Reinhardt domain $\Omega$ can be thought of as  
a holomorphic function on a different domain $\wt{\Omega}$, and this simplifies its study.

\subsection{Model domains as branched covers} 
 A map $\phi: \cx^n \to \cx^n$ will be said to be of the \emph{diagonal type} if there are 
positive integers $\ell_1,\dots, \ell_n$ such that 
\begin{equation}
    \label{eq-standard}
    \phi(z_1,\dots, z_n)=\left(z_1^{\ell_1},\dots, z_n^{\ell_n}\right). 
\end{equation}
\begin{prop} \label{prop-propermap} Let $n\geq 2$ and let $H$ be an elementary Reinhardt domain in 
$\cx^n$ of signature $1\leq s\leq n-1$. Then there is a proper holomorphic map of diagonal type from the model elementary domain $\Omega_{n,s}$ of \eqref{eq-Omegans} to $H$.
\end{prop}
\begin{proof} Let $k=(k_1,\dots,k_s, -k_{s+1},\dots, -k_n)$ be the multi-index such that $H=\mathscr{H}(k)$.
Let us set
$ K= \mathrm{lcm}(k_1,\dots, k_n),$
and let $\ell_j = \dfrac{K}{k_j}.$ Define the map $\phi$ by \eqref{eq-standard}. Then $\phi$ defines a proper holomorphic map from $\cx^n$ to itself.
To show that $\phi$ restricts to a proper map from $\Omega_{n,s}$ 
to $H$, it suffices to show that $\phi^{-1}(H)=\Omega_{n,s}.$ Indeed, if $z\in \cx^n$ is such that $\phi(z)\in H$, then we have $\abs{\phi(z)^k}<1.$
But since
\[ \phi(z)^k =(z_1^{\ell_1})^{k_1}\dots(z_s^{\ell_s})^{k_s} (z_{s+1}^{\ell_{s+1}})^{-k_{s+1}} \dots (z_{n}^{\ell_{n}})^{-k_{n}}=(z_1\dots z_s)^K(z_{s+1}\dots z_{n})^{-K}, \]
it follows that $z\in \Omega_{n,s}$ and the result follows. 
 \end{proof}

\begin{defn}
Let $H$ be an elementary Reinhardt domain in $\cx^n$ of signature $s$. The map $\phi:\Omega_{n,s}\to H$ given by \eqref{eq-standard} will be referred to
as the \emph{standard proper map} associated with $H$.
\end{defn}
Note that there may also be proper holomorphic maps from $\Omega_{n,s}$  to $H$ different from the standard map. And for certain elementary Reinhardt domains, biholomorphic maps can even be found. For example, the map from $\Omega_{2,1}=\{\abs{z_1}<\abs{z_2}<1\}\subset\cx^2$ to 
$\mathscr{H}(m,-n)=\{\abs{z_1}^{\frac{m}{n}}<\abs{z_2}<1\}\subset\cx^2 $ given by $(z_1, z_2)\to (z_1z_2^{n-1}, z_2^m)$ is a proper holomorphic map different from the standard map, and a biholomorphism if and only if $m=1$.
%%%%%%%%%%%%%%%%%%%%%%%%%%%%%%%%%%%%%%%%%%%%
\section{Norms of monomials}

In the following theorem, we describe the coefficients of the series expansion
\eqref{eq-bergmanseries} of the Bergman kernel of an elementary Reinhardt domain.

\begin{thm}\label{thm-monomial}
Let $n\geq 2$, let $1\leq s\leq n-1$, and let $\alpha\in \mathbb{Z}^n$. Let $\beta\in \mathbb{Z}^n$ be the multi-index $(\beta_1,\dots, \beta_n)$ such that 
\[ \beta_j=\alpha_j+1.\]
Then, on the model domain $\Omega_{n,s}$, we have
\begin{enumerate}
    \item $\norm{e_\alpha}^2_{\Omega_{n,s}} <\infty$ if and only if
\begin{equation}
    \label{eq-somegank}
    \beta_j>0\quad \text{and} \quad \beta_j+\beta_{\ell}>0 \quad \text{for $1\leq j\leq s$, and $s+1\leq \ell \leq n$. }
\end{equation} 
\item if $\alpha$ is such that $\norm{e_\alpha}^2_{\Omega_{n,s}} <\infty$ we have that
\begin{equation}
    \label{eq-RSdef}\norm{e_\alpha}^2_{\Omega_{n,s}} = \pi^n\cdot \frac{R_{n,s}(\beta)}{S_{n,s}(\beta)}, 
\end{equation}
where $R,S$ are homogeneous polynomials in $n$ variables with integer coefficients, with
\begin{equation}\label{eq-pmkhyp}
S_{n,s}(\beta)= \prod_{j=1}^s \beta_j \prod_{\substack{1\leq j \leq s\\ s+1 \leq \ell \leq n}}(\beta_j+\beta_\ell),\end{equation}
and $R_{n,s}$ is a homogeneous polynomial of total degree  $(n-s)(s-1)$ such that $R_{n,s}$ and $S_{n,s}$ have no common factors. We further have $R_{n,1}=1.$
\end{enumerate}
\end{thm}
Recall that the total degree of a monomial is the sum of exponents of each of the variables. 

\begin{proof}
Fix an $s\geq 1$, and we prove this by induction on $n$. We will start with the base case of $n=s$, which is not part of the statement of the theorem as stated, but for which the result also holds.
Denote by $\mathbb{D}^s$ the unit polydisc $\{\abs{z_j}<1, j=1,\dots, s\}$ in $\mathbb{C}^s$. Notice that  
\[ \Omega_{s,s}= \{z\in \mathbb{D}^s\mid\abs{z_1z_2\dots z_n}<1\}= \mathbb{D}^s. \]
In this case we have for $\alpha\in \mathbb{Z}^s$ by direct
computation that $\norm{e_\alpha}^2 <\infty$ if and only if
\[  \beta_j=\alpha_j+1>0,\quad j=1,\dots, s,\]
and for such $\alpha$
\begin{equation}\label{eq-polydisccoeftt}
\norm{e_\alpha}^2_{\mathbb{D}^s}= \pi^{s}\frac{1}{(\alpha_{1}+1)\dots(\alpha_{s}+1)}=\pi^{s}\frac{1}{
\beta_1\dots \beta_s}. \end{equation}
Therefore \eqref{eq-somegank} is satisfied, and if we take $R_{s,s}=1$ and $S_{s,s}=\beta_1\dots \beta_s $
then  \eqref{eq-pmkhyp} is satisfied, and $R_{s,s}$ does have
degree $(n-s)(s-1)=(s-s)(s-1)=0$, as needed.

We now proceed by induction. Assume the result is true for some 
$n\geq s$. For simplicity of notation, let $\mathbf{1}=(1,1,\dots, 1),$ then we set 
\begin{equation} \label{eq-d}
     \mathcal{D}_{n,s}(\beta) = \frac{1}{\pi^n}\norm{e_{\beta -\mathbf{1}}}^2_{\Omega_{n,s}},
 \end{equation}
and for $\beta_{n+1}\in \mathbb{Z}$ denote by $(\beta,\beta_{n+1})\in \mathbb{Z}^{n+1}$ the multi-index \[
(\beta, \beta_{n+1})=(\beta_1,\dots\beta_n,\beta_{n+1}).\]
To abbreviate the formulas that follow, let $\beta^*\in \mathbb{Z}^n$ be the multi-index given by 
\begin{equation}
    \label{eq-betastar}
     \beta_j^*=
    \begin{cases}
   \beta_j+\beta_{n+1} & \text{ if $1\leq j \leq s$}\\
   \beta_j-\beta_{n+1} & \text{ if $s+1 \leq j \leq n$.}
   \end{cases}
\end{equation}
Notice that $\beta^*$ actually depends on $\beta\in \mathbb{Z}^n$ and $\beta_{n+1}\in \mathbb{Z}$, though this has 
been suppressed from the notation.   We claim that $\mathcal{D}_{n+1,k}(\beta,\beta_{n+1})$ can be represented as follows
 \begin{equation}\label{eq-drecursion}
     \mathcal{D}_{n+1,s}(\beta,\beta_{n+1}) = \frac{1}{\beta_{n+1}}(\mathcal{D}_{n,s}(\beta)-\mathcal{D}_{n,s}(\beta^*)).
 \end{equation}
We postpone the proof of the claim to proceed with the induction.
Note that $\norm{e_{\beta - \mathbf{1}}}<\infty$ is equivalent to $\mathcal{D}_{n,s}(\beta)<\infty$.  
Let $(\beta, \beta_{n+1})\in \mathbb{Z}^n$. From \eqref{eq-drecursion},
it follows that $ \mathcal{D}_{n+1,s}(\beta,\beta_{n+1})<\infty$
if and only if $\mathcal{D}_{n,s}(\beta)<\infty$ and $\mathcal{D}_{n,s}(\beta^*)<\infty$, since each of $\mathcal{D}_{n,s}(\beta)$ 
and $\mathcal{D}_{n,s}(\beta^*)$ is strictly positive. From $\mathcal{D}_{n,s}(\beta)<\infty$, using the induction hypothesis, we see that the conditions \eqref{eq-somegank} hold. From $\mathcal{D}_{n,s}(\beta^*)<\infty$ we get the conditions
\[ \beta^*_j>0\quad \text{and} \quad \beta^*_j+\beta^*_{\ell}>0 \quad \text{for $1\leq j\leq s$, and $s+1\leq \ell \leq n$, } \]
which, using the definition of $\beta_j^*$ in \eqref{eq-betastar} becomes
\begin{equation}
    \label{eq-somegankstar}
    \beta_j+\beta_{n+1}>0, \quad \text{and} \quad \beta_j+\beta_\ell>0, 
    \quad 1 \leq j\leq s,\text{ and } s+1\leq \ell \leq n+1.
\end{equation}
Now \eqref{eq-somegank} and \eqref{eq-somegankstar} together imply that 
the conclusion (1) of the theorem we are proving holds for $n+1$, provided it holds for $n$.

Assuming now that  $ \mathcal{D}_{n+1,s}(\beta,\beta_{n+1})<\infty$,
by \eqref{eq-drecursion}, and the induction hypothesis,  we have that 
\begin{align}
        \mathcal{D}_{n+1,s}(\beta,\beta_{n+1}) &= \frac{1}{\beta_{n+1}}\left(\frac{R_{n,s}(\beta)}{S_{n,s}(\beta)} - \frac{R_{n,s}(\beta^*)}{S_{n,s}(\beta^*)}\right)\nonumber\\
        &= \frac{1}{\beta_{n+1}}\left(\frac{R_{n,s}(\beta)S_{n,s}(\beta^*) - R_{n,s}(\beta^*)S_{n,s}(\beta)}{S_{n,s}(\beta)S_{n,s}(\beta^*)}\right).\label{eq-gamma}
    \end{align}
Using the definition \eqref{eq-betastar} of $\beta^*$, we have
   \begin{align}
       S_{n,s}(\beta^*)&= \prod_{j=1}^s \beta_j^* \prod_{\substack{1\leq j \leq s\\ s+1 \leq \ell \leq n}}(\beta^*_j+\beta^*_\ell)\nonumber
       = \prod_{j=1}^s (\beta_j+\beta_{n+1}) \prod_{\substack{1\leq j \leq s\\ s+1 \leq \ell \leq n}}(\beta_j+\beta_\ell)\nonumber \\
       &=\prod_{\substack{1\leq j \leq s\\ s+1 \leq \ell \leq n+1}}(\beta_j+\beta_\ell). \label{eq-alpha}
   \end{align}
   Therefore using \eqref{eq-pmkhyp} and \eqref{eq-alpha}:
   \begin{align}
       S_{n,s}(\beta)S_{n,s}(\beta^*)=  &\prod_{j=1}^s \beta_j \prod_{\substack{1\leq j \leq s\\ s+1 \leq \ell \leq n}}(\beta_j+\beta_\ell) \prod_{\substack{1\leq j \leq s\\ s+1 \leq \ell \leq n+1}}(\beta_j+\beta_\ell)\nonumber\\
       = &S_{n+1,s}(\beta, \beta_{n+1})\cdot \prod_{\substack{1\leq j \leq s\\ s+1 \leq \ell \leq n}}(\beta_j+\beta_\ell), 
       \label{eq-beta}
   \end{align}
   where $S_{n+1,s}(\beta, \beta_{n+1})$ is as in \eqref{eq-pmkhyp}.  The expression in the numerator of \eqref{eq-gamma} is given, using \eqref{eq-alpha} and \eqref{eq-pmkhyp} by
   \begin{align}
      &R_{n,s}(\beta)S_{n,s}(\beta^*) - R_{n,s}(\beta^*)S_{n,s}(\beta)\nonumber\\
      &=  R_{n,s}(\beta)\cdot \prod_{\substack{1\leq j \leq s\\ s+1 \leq \ell \leq n+1}}(\beta_j+\beta_\ell)  - R_{n,s}(\beta^*)\cdot\prod_{j=1}^s \beta_j \prod_{\substack{1\leq j \leq s\\ s+1 \leq \ell \leq n}}(\beta_j+\beta_\ell) \nonumber\\
      &=\left( R_{n,s}(\beta)\cdot \prod_{j=1}^s (\beta_j+\beta_{n+1}) - R_{n,s}(\beta^*)\cdot \prod_{j=1}^s \beta_j \right)\cdot\prod_{\substack{1\leq j \leq s\\ s+1 \leq \ell \leq n}}(\beta_j+\beta_\ell). \label{eq-delta}
   \end{align}
     Using \eqref{eq-beta} and \eqref{eq-delta} in \eqref{eq-gamma}, we see that the numerator and denominator of \eqref{eq-gamma} share the
   common factor $\prod_{\substack{1\leq j \leq s\\ s+1 \leq \ell \leq n}}(\beta_j+\beta_\ell)$.  Removing this common factor we see that 
   \[\mathcal{D}_{n+1}(\beta, \beta_{n+1})= \frac{f(\beta_{n+1})/{(\beta_{n+1})}}{S_{n+1,s}(\beta, \beta_{n+1})},\]
   where we now think of $(\beta_1,\dots, \beta_{n+1})$ as indeterminates, and $f$ as a polynomial in the ring 
   $\mathbb{Q}(\beta_1,\dots, \beta_{n})[\beta_{n+1}]$ of polynomials in the indeterminate $\beta_{n+1}$ over the field
   of rational functions $\mathbb{Q}(\beta_1,\dots, \beta_{n})$ in $n$ indeterminates, with $f$ given by
   \[f(\beta_{n+1})= R_{n,s}(\beta)\cdot \prod_{j=1}^s (\beta_j+\beta_{n+1}) - R_{n,s}(\beta^*)\cdot \prod_{j=1}^s \beta_j. \]
   Now the formulas \eqref{eq-betastar} defining $\beta^*$ in terms of $\beta_1, \dots, \beta_{n+1}$ show that if $\beta_{n+1}=0$,
   then $\beta^*=\beta$. It now follows that $f(0)=0$, so that \[ R_{n+1,s}(\beta, \beta_{n+1})=f(\beta_{n+1})/{\beta_{n+1},} \] is a polynomial in the ring $\mathbb{Q}(\beta_1,\dots, \beta_{n})$. But noting further that $f\in \mathbb{Z}[\beta_1,\dots, \beta_{n+1}]$,
   and the divisor $\beta_{n+1}$ has leading coefficient 1, we see that in fact $R_{n+1,s}\in \mathbb{Z}[\beta_1,\dots, \beta_{n+1}]$
   which we wanted to prove. We therefore have  the recursive formula:
   
\begin{equation}
\label{eq-Rrecursive}
   R_{n+1,s}(\beta, \beta_{n+1}) =\frac{1}{\beta_{n+1}}\cdot \left( R_{n,s}(\beta)\cdot \prod_{j=1}^s (\beta_j+\beta_{n+1}) - R_{n,s}(\beta^*)\cdot \prod_{j=1}^s \beta_j \right).
\end{equation}
By the induction hypothesis, $R_{n,s}$ is a homogeneous polynomial in the $n$ variables $\beta_1,\dots, \beta_n$ of total degree $(n-s)(s-1)$. By the definition \eqref{eq-betastar} of 
$\beta^*$, we see that $R_{n,s}(\beta^*)$ is a also homogeneous polynomial 
of the $n+1$ variables $\beta_1,\dots, \beta_{n+1}.$ The quantity 
in large parentheses in \eqref{eq-Rrecursive} is therefore the difference 
of two homogeneous polynomials of total degree $(n-s)(s-1)+s$. It is therefore either zero, or itself a homogeneous polynomial of degree $(n-s)(s-1)+s$. But it cannot be zero, since then the norm of a monomial is zero, which is absurd. 
 Finally by \eqref{eq-Rrecursive}, the polynomial $R_{n+1,s}$ is also homogeneous, being the ratio of two homogeneous polynomials, and has total degree
 \[ (n-s)(s-1)+s-1= \left((n+1)-s\right)(s-1).\]
We will now show that $R_{n+1,s}(\beta)$ and $S_{n+1,s}(\beta)$ have no common factors.

By induction hypothesis  $S_{n,s}(\beta)$ has no common factors with $R_{n,s}(\beta)$. Since $S_{n,s}(\beta)$ is a product of  linear factors $\beta_{j}$ and $(\beta_{j} + \beta_{\ell})$ where $1 \leq j \leq s$, and $s+1 \leq \ell \leq n$, none of these factors divide $R_{n,s}(\beta)$.

From the symmetry of $\Omega_{n+1,s}$, the definition \eqref{eq-RSdef}, and the symmetry of $S_{n+1,s}$ we know that $R_{n+1,s}(\beta,\beta_n+1)$ is symmetric in variables $\beta_{1}, \dots, \beta_{s}$ and variables $\beta_{s+1}, \dots, \beta_{n+1}$. Starting from these facts, we can verify that none of the linear factors of $S_{n+1,s}$ divides the right hand side of \eqref{eq-Rrecursive}, by noting that even if these linear factors vanish, the right hand side of \eqref{eq-Rrecursive} does not.  Hence, $R_{n+1, s}(\beta , \beta_{n+1})$ and $S_{n+1,s}(\beta , \beta_{n+1})$ have no common factors.

Therefore, the inductive proof of the Theorem is complete, except that 
we need to establish the claim \eqref{eq-drecursion} on which the above 
 induction was based. Note that from \eqref{eq-d}, with $\mathbf{1}=(1,\dots,1)$, we have
 \[\mathcal{D}_{n,s}(\beta)=\frac{1}{\pi^n} \norm{e_{\beta-\mathbf{1}}}^2 = \frac{1}{\pi^n} \int_{\Omega_{n,s}}\abs{e_{\beta-\mathbf{1}}(z)}^2 dV(z). \]
 Using polar coordinates $z_j=r_je^{i\theta_j}$ and using the fact that 
 $dV(z)= \prod_{j=1}^n r_j dr_jd\theta_j =r^{\mathbf{1}}dV(r)dV(\theta)$, where $r=(r_1,\dots,r_n)$, we have
 \[\mathcal{D}_{n,s}(\beta)= \frac{1}{\pi^n}\cdot (2\pi)^n \int_{\abs{\Omega_{n,s}}} r^{2\beta-\mathbf{1}}dV(r), \]
 where $\abs{\Omega_{n,s}}\subset \rl^n$ is the Reinhardt shadow of $\Omega_{n,s}$, i.e.,
 the image of $\Omega_{n,s}$ under the map $z\mapsto (\abs{z_1}, \dots, \abs{z_n})$. We will make the further change of variables  $t_j=r_j^2$, which maps $\abs{\Omega_{n,s}}$ diffeomorphically to itself. The integral now takes the form:
 \[\mathcal{D}_{n,s}(\beta)= \frac{1}{\pi^n}\cdot (2\pi)^n \int_{\abs{\Omega_{n,s}}} t^{\beta-\mathbf{1}}dV(t).  \]
 We will transform this integral into an $n$-fold repeated integral.
For simplicity of notation we  denote   repeated integrals with differential in front and integrand after that, so that
\[ \int_{x_2=a_2}^{b_2} g(x_2)\left(\int_{x_1=a_1}^{b_1} f(x_1, x_2) dx_1\right)dx_2=
\int_{a_2}^{b_2} dx_2 \cdot g(x_2)\int_{a_1}^{b_1} dx_1\cdot f(x_1, x_2),\]
and adopt similar notations for multiple repeated integrals, so that the innermost integral in the conventional notation  is  the rightmost factor. The region of integration over which $t\in \rl^n$ ranges is described by the inequalities
\[ 0 \leq t_{1}\dots t_{s}<t_{s+1}\dots t_{n} < 1,\, 0\leq  t_{1}<1, \dots, 0\leq t_{n} < 1.\]
Then, $\mathcal{D}_{n,s}(\beta)$  can be expressed explicitly by the following $n$-fold integral:
\[ \int_0^1dt_1\cdot t_1^{\beta_1 - 1}\int_0^1 dt_2\cdot t_2^{\beta_2-1}\dots\int_0^1dt_s\cdot t_s^{\beta_s-1}\int_{t_1\dots t_s}^1 dt_{s+1}\cdot t_{s+1}^{\beta_{s+1}-1}\int_{\frac{t_1\dots t_s}{t_{s+1}}}^1 dt_{s+2}\cdot t_{s+2}^{\beta_{s+2}-1}\dots\int_{\frac{t_1\dots t_s}{t_{s+1}\dots t_{n-1}}}^1 dt_n\cdot t_{n}^{\beta_{n}-1}.\]

Similarly, 
\begin{align*}
    \mathcal{D}_{n+1,s}(\beta,\beta_{n+1}) &= \int_0^1dt_1 \cdot t_1^{\beta_1-1}\cdots\int_0^1 dt_s t_s^{\beta_s-1}\int_{t_1\dots t_s}^1dt_{s+1}t_{s+1}^{\beta_{s+1}-1}\int_{\frac{t_1\dots t_s}{t_{s+1}}}^1dt_{s+2} t_{s+2}^{\beta_{s+2}-1}\cdots\\
&\cdots\int_{\frac{t_1\dots t_s}{t_{s+1}\dots t_{n-1}}}^1dt_nt_{n}^{\beta_{n}-1}\int_{\frac{t_1\dots t_s}{t_{s+1}\dots t_{n}}}^1dt_{n+1}t_{n+1}^{\beta_{n+1}-1}\\
%%%%%%%%%%%%%%%%%%%%%%%%%
&=\int_0^1dt_1  t_1^{\beta_1-1}\cdots\int_0^1dt_s t_s^{\beta_s-1}\int_{t_1\dots t_s}^1dt_{s+1}t_{s+1}^{\beta_{s+1}-1}\int_{\frac{t_1\cdot t_s}{t_{s+1}}}^1dt_{s+2}\cdot t_{s+2}^{\beta_{s+2}-1}\cdots \\
&\cdots\int_{\frac{t_1\dots t_s}{t_{s+1}\dots t_{n-1}}}^1dt_n\cdot\frac{1}{\beta_{n+1}}
\left(1- \left(\frac{t_1\dots t_s}{t_{s+1}\dots t_{n}}\right)^{\beta_{n+1}}\right),\\
&\text{(where we have evaluated the innermost integral)}\\
&= \frac{1}{\beta_{n+1}}\left(\mathcal{D}_{n,s}(\beta) - \vphantom{\int_{\frac{t_1\dots t_s}{t_{s+1}\dots t_{n-1}}}^1 dt_n\cdot t_{n}^{\beta_{n}-\beta_{n+1}-1}}\right.\\
&\int_0^1 dt_1\cdot t_1^{\beta_1+\beta_{n+1}-1}\cdots\int_0^1dt_s t_s^{\beta_s+\beta_{n+1}-1}
\int_{t_1\dots t_s}^1 dt_{s+1}t_{s+1}^{\beta_{s+1}-\beta_{n+1}-1}\int_{\frac{t_1\dots t_s}{t_{s+1}}}^1dt_{s+2} t_{s+2}^{\beta_{s+2}-\beta_{n+1}-1}\\
&\left.\cdots\int_{\frac{t_1\dots t_s}{t_{s+1}\dots t_{n-1}}}^1 dt_n\cdot t_{n}^{\beta_{n}-\beta_{n+1}-1}\right)\\
&= \frac{1}{\beta_{n+1}}\left( \mathcal{D}_{n,s}(\beta)- \mathcal{D}_{n,s}(\beta^*)\right),
\end{align*}
which completes the proof of \eqref{eq-drecursion}.
\end{proof}

\subsection{Proof of Theorem~\ref{thm-notrational}}
From \eqref{eq-bergmanseries}, we may write
\[ \mathbb{B}_{\Omega_{n,n-1}}(z,w)= \sum_{\beta\in \mathcal{T}} \frac{1}{\norm{e_{\beta-\mathbf{1}}}^2} t^{\beta - \mathbf{1}}, \]
where $\mathbf{1}=(1,\dots, 1)$ and  $\mathcal{T}$ is the set of indices corresponding to $s=n-1$ in \eqref{eq-somegank}, i.e.
\[ \beta_j>0, \beta_j+\beta_n>0, \text{ for } 1 \leq j \leq n-1.\]
Furthermore, we have from Theorem~\ref{thm-monomial} that 
\[ \norm{e_{\beta-\mathbf{1}}}^2 =  \pi^n \frac{R_{n, n-1}(\beta)}{S_{n, n-1}(\beta)},\]
where by \eqref{eq-pmkhyp}, we have
\[ S_{n,n-1}(\beta)=\prod_{j=1}^{n-1}\beta_j \prod_{j=1}^{n-1} (\beta_j+\beta_n) = \prod_{j=1}^{n-1} \beta_j(\beta_j+\beta_n), \]
and using the recursive relation \eqref{eq-Rrecursive} and the fact that 
$R_{n,n}\equiv 1$  (see \eqref{eq-polydisccoeftt}), we see that
\[ R_{n,n-1}(\beta)= \frac{1}{\beta_n} \left( \prod_{j=1}^{n-1} (\beta_j+\beta_n) - \prod_{j=1}^{n-1}\beta_j \right).\]
Therefore, with $t_j=z_j \ol{w_j}$, we have 
\begin{align}
    \mathbb{B}_{\Omega_{n,n-1}}(z,w)&=\wt{B}(t_1,\dots, t_n)\nonumber\\
    &= \frac{1}{\pi^n}\sum_{\beta\in \mathcal{T}} \frac{\prod_{j=1}^{n-1} \beta_j(\beta_j+\beta_n)}{\frac{1}{\beta_n} \left( \prod_{j=1}^{n-1} (\beta_j+\beta_n) - \prod_{j=1}^{n-1}\beta_j \right)}t^{\beta -\mathbf{1}}.    \label{eq-kernel}
\end{align}
We now consider the function $\wt{b}$ of one variable defined by
\[ \wt{b}(t_n)=\wt{B}(\underbrace{0,\dots,0}_{n-1}, t_n).\]
This is defined in the punctured disc $\{0<\abs{t_n}<1\},$
and noting that in \eqref{eq-kernel} only the terms with $\beta_j=1$, $1\leq j \leq n-1$ survive if $t_1=\dots=t_{n-1}=0$, we conclude that
\begin{align}
    \wt{b}(t_n)&= \sum_{\beta_n=0}^\infty \frac{\beta_n(1+\beta_n)^{n-1}}{(1+\beta_n)^{n-1}-1} t_n^{\beta_n-1}\nonumber\\
    &=\frac{ t_n^{-1}}{n-1}+ \sum_{k=1}^\infty k t_n^{k-1} + \sum_{k=1}^\infty \frac{k}{(k+1)^{n-1} -1} t_n^{k-1}\nonumber\\
    &=\frac{ t_n^{-1}}{n-1}+  \frac{1}{(1-t_n)^2} + \widehat{b}(t_n) \label{eq-whb},
\end{align} 
where 
\[ \widehat{b}(t_n)= \sum_{k=1}^\infty \frac{k}{(k+1)^{n-1} -1} \cdot t_n^{k-1}.\]

Since the function $\wt{B}$ is holomorphic on the domain $\{(z_1\ol{w_1},\dots, z_n\ol{w_n})\mid z,w\in \Omega_{n, n-1}\}$ 
it follows that $\wt{b}$ is holomorphic in the punctured disc $\{0<\abs{t_n}<1\},$ and therefore $\widehat{b}$ is holomorphic in the unit 
disc $\{\abs{t_n}<1\}$.

Now for a contradiction, assume that $\mathbb{B}_{\Omega_{n,n-1}}$ is rational. It follows that $\widehat{b}$ is a rational function of one variable, holomorphic in the unit disc, and its $k$-th Taylor coefficient 
decays as $k^{-(n-2)}$ as $k\to\infty$. Recall that by hypothesis $n\geq 3$, 
so the coefficients go to zero. 

Let $\alpha_1, \dots, \alpha_m$ be the poles of the rational function $\widehat{b}$, where $\abs{\alpha_j}\geq 1$ since $\widehat{b}$ is 
holomorphic in the unit disc. It follows by expansion in partial fractions (see \cite[p. 256ff]{analytic}) that the $k$-th Taylor coefficient 
of $\widehat{b}$ is of the form $\sum_{j=1}^m \alpha_j^{-k} \Pi_j(k)$ where $\Pi_j$ is a polynomial for each $j$. 
Since the coefficients go to zero as $k\to \infty$, we must have $\abs{\alpha_j}>1$, for each $j=1,\dots, m$. 
Therefore, the decay of the coefficients is exponential in $k$, which contradicts the $k^{-(n-2)}$ decay. Therefore $\widehat{b}$ cannot be a rational function, and so $\mathbb{B}_{\Omega_{n,n-1}}$ is not a rational function if $n\geq 3$.

\subsection{Some remarks on the nature of the Bergman Kernel of \texorpdfstring{$\Omega_{n,s}$}{Omegans}}\label{sec-bergmanremarks}
The form of the coefficients of the series in Theorem~\ref{thm-monomial} as 
well as the argument in the proof of Theorem~\ref{thm-notrational} 
suggest that the Bergman kernel of $\Omega_{n,s}$ is not rational except for
$s=1$, though we do not have a complete proof of this yet. However, Theorem~\ref{thm-monomial} is already sufficient to rule out certain
hasty conjectures about the form of $\mathbb{B}_{\Omega_{n,s}}$ that one might make based on \eqref{eq-kscrh} or similar formulas in \cite{park}. For example, for $s\not=1$, the kernel $\mathbb{B}_{\Omega_{n,s}}$ \emph{cannot}  be written in the form
\[ \frac{1}{\pi^n} \frac{P(t)}{\left(\prod_{b=s+1}^n t_b^{k_b}- \prod_{a=1}^s t_a^{k_a} \right)^2\cdot \prod_{b=s+1}^n (1-t_b)^2}, \]
for a polynomial $P$, since the coefficient of $t^\alpha$ of the Taylor expansion of this function is a polynomial in $\alpha$. Additionally, we saw above that when $s\not =1$, the Taylor coefficients of $\mathbb{B}_{\Omega_{n,s}}$ are rational functions of $\alpha$ which are not polynomials. Another interesting algebraic property is given by the following:
\begin{prop}
Let $n\geq 2$ and $1\leq s\leq n-1$. Let $\wt{B}$ be the function of $t_j=z_j\ol{w_j}$ associated with the Bergman kernel of $\Omega_{n,s}$, 
as defined in \eqref{eq-btilde}. Then there is a nonzero linear differential operator 
$\mathscr{L}$ with polynomial coefficients, such that $\mathscr{L}\wt{B}$ is a polynomial.
\end{prop}
\begin{proof}
The case $s=1$ is trivial, since then by Theorem~\ref{thm-computation}, $\wt{B}$ is a rational function $P/Q$, where $P,Q$ are polynomials. 
Therefore we can simply take $\mathscr{L}$ to be the zeroth order multiplication operator determined by $Q$.

Notice that we can write, thanks to Theorem~\ref{thm-monomial}, the series 
representation
\[\wt{B}(t) =\frac{1}{\pi^n}\cdot\frac{1}{t_1\dots t_n}\sum_{\beta\in \mathcal{S}} \frac{S_{n,s}(\beta)}{R_{n,s}(\beta)}t^\beta, \]
where $R_{n,s}(\beta_1,\dots, \beta_n)$ and $S_{n,s}(\beta_1,\dots, \beta_n)$ are homogeneous polynomials in the variables $\beta_1,\dots,\beta_n$, and 
$\mathcal{S}$ is the subset of $\mathbb{Z}^n$ determined by the conditions \eqref{eq-somegank}. Let $\mathcal{M}$ denote the multiplication operator induced by the polynomial $t_1\dots t_n$, and let
\[\mathscr{L}_0=R_{n,s}\left( t_1\frac{\partial}{\partial t_1},\dots, t_n\frac{\partial}{\partial t_n}\right)\circ \mathcal{M}. \]
Then we see that 
\begin{equation}
\label{eq-lukenumber} \mathscr{L}_0\widetilde{B}(t)= \frac{1}{\pi^n} \sum_{\beta\in S}S_{n,s}(\beta)t^\beta. 
\end{equation}

Now since the coefficients  $S_{n,s}(\beta)$ are polynomials in $\beta$ and 
the region of summation $\mathcal{S}$ is the intersection of a finite number of \emph{closed} half-spaces in $\mathbb{Z}^n$ (since the open conditions in \eqref{eq-somegank} can be replaced by closed conditions), it follows that the right hand side of \eqref{eq-lukenumber} is a rational function (cf. the proof of Theorem~\ref{thm-computation} below). If $Q(t)$ is the denominator of this rational function, and $\mathcal{Q}$ is the multiplication operator induced by $Q$, we can take $\mathscr{L}=\mathcal{Q}\circ \mathscr{L}_0$.
\end{proof}

%%%%%%%%%%%%%%%%%%%%%%%%%%%%%%%%%%%%%%%%%%
\section{Proof of Theorem~\ref{thm-computation}}
\subsection{Kernel of Model domain}We begin by computing the Bergman kernel of the model elementary Reinhardt domain $\Omega_{n,1}$:
\begin{prop} \label{prop-H}
The Bergman kernel of $\Omega_{n,1}$ is given by
\[ \displaystyle{\mathbb{B}_{\Omega_{n,1}}(z,w)=  \frac{1}{\pi^n}\cdot\frac{\displaystyle{\prod_{b=2}^nt_b}}{\displaystyle{\left(\prod_{b=2}^n t_b-t_1\right)^2\cdot\prod_{b=2}^n\left(1-t_b\right)^2}}, }\]
where
\[ t_b= z_b\ol{w_b} \quad \text{ for } 1\leq b \leq n.\]
\end{prop}
\begin{proof}
From Theorem~\ref{thm-monomial}, we see that for $\alpha\in \mathbb{Z}^n$, we have 
$\norm{e_\alpha}_{\Omega_{n,1}}^2 <\infty$ if and only if
\[\alpha_1+1>0, \alpha_1+\alpha_\ell+2>0,\quad 2\leq \ell \leq n, \]
which is equivalent to
\[ \alpha_1\geq 0, \alpha_1+\alpha_\ell+1 \geq 0,\quad 2\leq\ell\leq n.\]
Let $\mathcal{S}\subset\mathbb{Z}^n$ be the set of multi-indices satisfying the above condition. Also from Theorem~\ref{thm-monomial}, it follows that
for $\alpha\in \mathcal{S}$ we have
\[\norm{e_\alpha}^2_{\Omega_{n,1}} = \pi^n \frac{1}{(\alpha_1+1)\prod_{b=2}^n (\alpha_1+\alpha_b+2)}.\]
Using \eqref{eq-bergmanseries} and the abbreviation $t_b=z_b \ol{w_b}$, we have by a direct summation of the series \eqref{eq-bergmanseries}:
\begin{align*}
    \mathbb{B}_{\Omega_{n,1}}(z,w)&= \frac{1}{\pi^n}\sum_{\alpha\in \mathcal{S}}\left((\alpha_1+1)\prod_{b=2}^n (\alpha_1+\alpha_b+2)\right)t^\alpha\nonumber\\
  &= \frac{1}{\pi^n}\cdot\sum_{\alpha_1 = 0}^\infty (\alpha_1+1)t_1^{\alpha_1}\prod_{b=2}^n \left(\sum_{\alpha_b=-\alpha_1-1}^\infty (\alpha_1+\alpha_b+2)t_b^{\alpha_b}\right) \nonumber\\
      &= \frac{1}{\pi^n}\cdot\prod_{b=2}^n\frac{1}{t_b(1-t_b)^2}\sum_{\alpha_1 = 0}^\infty (\alpha_1+1)t_1^{\alpha_1}\prod_{b=2}^n t_b^{-\alpha_1} \\
       &\left(\text{using the easily proved identity $\displaystyle{\sum_{\alpha_b = -\alpha_1-1}^\infty (\alpha_1+\alpha_b+2) t_{b}^{\alpha_b} = \frac{t_{b}^{-\alpha_1-1}}{(1-t_{b})^2}}$}\right)\\
    &= \frac{1}{\pi^n}\cdot\prod_{b=2}^n\frac{1}{t_b(1-t_b)^2}\sum_{\alpha_1 = 0}^\infty (\alpha_1+1)\rho^{\alpha_1} ,\quad \text{ with $\rho=\frac{t_1}{\prod_{b=2}^n t_b}$} \\
    &= \frac{1}{\pi^n} \cdot\frac{1}{(1-\rho)^2}\cdot\prod_{b=2}^n\frac{1}{t_b(1-t_b)^2} \\
    &= \frac{1}{\pi^n}\cdot\frac{1}{\left(1-\dfrac{t_1}{\prod_{b=2}^n t_b}\right)^2}\cdot\prod_{b=2}^n\frac{1}{t_b(1-t_b)^2} \\
    &= \frac{1}{\pi^n}\cdot\frac{\prod_{b=2}^n t_b}{(\prod_{b=2}^n t_b-t_1)^2\cdot\prod_{b=2}^n(1-t_b)^2},
    \end{align*}
where we have used the identity $\sum_{\alpha_1=0}^\infty (\alpha_1+1)\rho^{\alpha_1}= \frac{1}{(1-\rho)^2}$ which holds since $\abs{\rho}<1$.
\end{proof}

\subsection{Explicit Kernel} \label{sec-explicit}

The following simple arithmetical fact will be used:
\begin{lem}\label{lem-numbertheoretic}
    Let $k_1, \dots, k_n$ be positive integers such that $\mathrm{gcd}(k_1, \dots, k_n) = 1$, i.e. $k_1, \dots, k_n$ are relatively prime. Let $K = \mathrm{lcm}(k_1, \dots, k_n)$ and  $\ell_j = \dfrac{K}{k_j}$ with $1 \leq j \leq n$.
    Then \[ \mathrm{lcm}(\ell_1, \dots, \ell_n) = K.\]
\end{lem}

\begin{proof}
Let $\displaystyle{k_j = \prod_{p \in \text{Primes}} p^{v_{j}(p)}}$ be the prime factoring of $k_j$. Then $K =\displaystyle{ \prod_{p \in \text{Primes}} p^{N(p)}}$ where 
\[ N(p) = \max_{1\leq j\leq n}(v_{j}(p)).\]
Now 
\[\ell_j = \frac{K}{k_j} = \prod_{p \in \text{Primes}} p^{N(p) - v_{j}(p)}.\]

So, 
\[\mathrm{lcm}(\ell_1,\dots,\ell_n) = \prod_{p \in \text{Primes}} p^{{\max_j(N(p) - v_{j}(p))}} = \prod_{p \in \text{Primes}} p^{N(p) - \min_j(v_{j}(p))} = \prod_{p \in \text{Primes}} p^{N(p)} =K,\]
where we have used the fact that since $\mathrm{gcd}(k_1,\dots,k_n)=1$, it follows that 
$\displaystyle{ \min_{1\leq j \leq n}(v_{j}(p))=0.}$
\end{proof}
\begin{proof}[Proof of Theorem~\ref{thm-computation}]
Let $\phi:\Omega_{n,1}\to \mathscr{H}(k)$ be the standard proper holomorphic map which was constructed in Proposition~\ref{prop-propermap}. Notice that 
this map is given by the formula
\begin{equation}
        \label{eq-phidef}
        \phi(z_1,\dots, z_n)= \left(z_1^{\ell_1}, \dots, z_n^{\ell_n}\right),
\end{equation}
where $\ell_j$ has exactly the same meaning as in the statement of our result.
Now by the famous Bell transformation formula (\cite{belltransformation}):
\begin{equation}
    \label{eq-bell}
    u(z)\cdot\mathbb{B}_{\mathscr{H}(k)}(\phi(z),w) = \sum_{j}\mathbb{B}_{\Omega_{n,1}}(z,\Phi_j(w))\cdot \ol{U_j(w)},
\end{equation}
where $u=\det(\phi')$, the $\Phi_j$'s are local branches of $\phi^{-1}$, and
$U_j= \det(\Phi_j')$.  The Jacobian determinant of  $\phi$ is given by
\[u(z)= \det \phi'(z)= \det \mathrm{diag}(\ell_1z_1^{\ell_1-1},\cdots, \ell_n z_n^{\ell_n-1})=  \prod_{a=1}^n \ell_a z_a^{\ell_a-1}.\]

The map $\phi$  has $L =\prod_{a=1}^n \ell_a $ local inverses.
To enumerate them, introduce the set of multi-indices
\begin{equation}
    \label{eq-gothicb}
    \mathfrak{B}=\{(j_1,\dots,j_n)\in \mathbb{Z}^n\mid 0\leq j_a \leq \ell_a-1, \text{ for } a=1,\dots, n\},
\end{equation}
then for each multi-index $j\in \mathfrak{B}$, there is a branch $\Phi_j$ of the
local inverse of $\phi$ given by
\[\Phi_j(z_1,\cdots, z_n)= \left(\zeta_1^{j_1}z_1^{\frac{1}{\ell_1}},\zeta_2^{j_2}  z_2^{\frac{1}{\ell_2}}, \cdots,\zeta_n^{j_n} z_n^{\frac{1}{\ell_b}}\right), \]
where 
\[ \zeta_a = e^{\frac{2\pi i}{\ell_a}}, \quad\text{ for each } 1\leq a \leq n \]
is an $\ell_a$-th root of unity, and the root functions $z_1^\frac{1}{\ell_1}, \dots, z_n^\frac{1}{\ell_n}$ exist locally off the critical locus. We then 
have for each $j\in \mathfrak{B}$
\[ U_j(w)= {\det \Phi_j'(w)}= \det \mathrm{diag}\left(\frac{\zeta_1^{j_1}}{\ell_1}{w_1}^{\frac{1}{\ell_1}-1},\cdots,
\frac{\zeta_n^{j_n}}{\ell_n}{w_n}^{\frac{1}{\ell_n}-1}\right)= \prod_{a=1}^n \frac{{\zeta_a}^{j_a}}{\ell_a}{w_a}^{\frac{1}{\ell_a}-1} ,\]
where $\mathrm{diag}(\cdot)$ denotes a diagonal matrix with the specified diagonal entries. 
Therefore by Bell's formula  \eqref{eq-bell} we have
\begin{align*}
    & \prod_{a=1}^n \ell_a z_a^{\ell_a-1} \cdot \mathbb{B}_{\mathscr{H}(k)}(\phi(z), w)
   =\sum_{j\in \mathfrak{B}}\, \mathbb{B}_{\Omega_{n,1}}(z,\Phi_j(w)) \cdot\ol{ \prod_{a=1}^n \frac{{\zeta_a}^{j_a}}{\ell_a}{w_a}^{\frac{1}{\ell_a}-1}}\\
    &= \frac{1}{\pi^n} \cdot \sum_{j\in \mathfrak{B}}
    \dfrac{\displaystyle{\prod_{b=2}^n
    \ol{\zeta_b}^{j_b}z_b\ol{w_b}^\frac{1}{\ell_b}}}
    {\displaystyle{\left(\prod_{b=2}^n {\ol{\zeta_b}}^{j_b}z_b\ol{w_b}^\frac{1}{\ell_b}
    -\zeta_1^{j_1}z_1\ol{w_1}^{\frac{1}{\ell_1}}\right)^2 \cdot
    \prod_{b=2}^n\left(1- \ol{\zeta_{b}}^{j_b}z_bw_b^{\frac{1}{\ell_b}}\right)^2}}\cdot 
    \prod_{a=1}^n \frac{\ol{\zeta_a}^{j_a}}{\ell_a}\,\ol{w_a}^{\frac{1}{\ell_a}-1},
\end{align*}
where  we have used the formula in Proposition~\ref{prop-H} for the Bergman kernel of $\Omega_{n,1}$. Introduce the abbreviations
\begin{equation}
    \label{eq-ra}r_a= z_a \ol{w}_a^\frac{1}{\ell_a}, \quad a=1,\dots, n,
\end{equation}
so that we have from the above (recall that $L=\prod_{j=1}^n \ell_j$)
\begin{align}
  \mathbb{B}_{\mathscr{H}(k)}(\phi(z), w)& = \frac{1}{\pi^n L^2}
  \sum_{j\in \mathfrak{B}}
  \frac{\displaystyle{\prod_{a=1}^n \ol{\zeta_a}^{j_a}r_a^{1-\ell_a}\cdot \prod_{b=2}^n\ol{\zeta_b}^{j_b}r_b }}
  {\displaystyle{\left(\prod_{b=2}^n \ol{\zeta_b}^{j_b}r_b -\ol{\zeta_1}^{j_1}r_1\right)^2 \prod_{b=2}^n (1-\ol{\zeta_b}^{j_b}r_b)^2}} \nonumber \\
  &=\frac{1}{\pi^n L^2}
  \sum_{j\in \mathfrak{B}}
  \frac{\displaystyle{\ol{\zeta_1}^{j_1}r_1^{1-\ell_1}\cdot \prod_{b=2}^n\ol{\zeta_b}^{2j_b}r_b^{2-\ell_b} }}
  {\displaystyle{\left(\prod_{b=2}^n \ol{\zeta_b}^{j_b}r_b -\ol{\zeta_1}^{j_1}r_1\right)^2 \prod_{b=2}^n (1-\ol{\zeta_b}^{j_b}r_b)^2}}. \label{eq-whb}  
%  &= \widehat{B}(r_1,\dots,r_n), \nonumber
\end{align}
Let $\widehat{B}(r_1,\dots,r_n)$  denote the quantity in \eqref{eq-whb}.
 We claim that the function  $\widehat{B}$ of $n$ variables has the following invariance property, which will be needed later: for each $c$ with $1\leq c \leq n$, we have 
\begin{equation}
    \widehat{B}(r_1,\cdots, \ol{\zeta_c}r_c,  \cdots,r_n) = \widehat{B}(r_1,\cdots, r_n)  \label{eq-invariance}.
\end{equation}
To see this, notice that we have, for each $c$ with $2\leq c \leq n$, that
\begin{align*} 
 &\widehat{B}(r_1,r_2,\cdots, \ol{\zeta_c}r_c,  \cdots,r_n)=\\ &\frac{1}{\pi^n L^2}
  \sum_{j\in \mathfrak{B}}
  \frac{\displaystyle{\ol{\zeta_1}^{j_1}r_1^{1-\ell_1}\cdot \ol{\zeta_c}^{2(j_c+1)}r_c^{2-\ell_c} \cdot \prod_{\substack{2\leq b\leq n\\ b\not=c}}\ol{\zeta_b}^{2j_b}r_b^{2-\ell_b} }}
  {\displaystyle{\left(\ol{\zeta_c}^{j_c + 1}r_c\prod_{\substack{2\leq b\leq n\\ b\not=c}} \ol{\zeta_b}^{j_b}r_b -\ol{\zeta_1}^{j_1}r_1\right)^2 (1-\ol{\zeta_c}^{j_c+1}r_c)^2\prod_{\substack{2\leq b\leq n\\ b\not=c}} (1-\ol{\zeta_b}^{j_b}r_b)^2}} .\nonumber\\
 \end{align*}
 Notice that the above sum is precisely the same as $\widehat{B}(r_1,...,r_n)$, since changing $j_c$ to $j_c+1$ simply amounts to a re-indexing of the sum, thanks to the fact that the $\ell_c$-th roots of unity form a cyclic group generated by $\zeta_c$. 

 In a similar way, $\widehat{B}(\ol{\zeta_1}r_1,r_2,\cdots,r_n)$ is precisely the same as $\widehat{B}(r_1,...,r_n)$, since changing $j_1$ to $j_1+1$ simply amounts to a re-indexing of the sum, thanks to the fact that the $\ell_1$-th roots of unity form a cyclic group generated by $\zeta_1$.  These two observations combined establish \eqref{eq-invariance}. 
 
  Now let
 \begin{equation}
     \label{eq-Delta}\Delta = \left(\left(\prod_{b=2}^n r_b\right)^K-r_1^K\right)^2\cdot \prod_{b=2}^n\left(1-r_b^{\ell_b}\right)^2,
 \end{equation}
where $K= \mathrm{lcm}(k_1,\dots, k_n)$ as in  the statement of the theorem.
Then we can write
\begin{align}
&\widehat{B}(r_1,\dots, r_n)\nonumber\\
 &=  \frac{1}{\pi^n  L^2 \Delta}\,
   \sum_{j\in \mathfrak{B}}\left(
   \ol{\zeta_1}^{j_1}r_1^{1-\ell_1}\prod_{b=2}^n \ol{\zeta_b}^{2j_b}r_b^{2-\ell_b}
   \cdot \frac{\displaystyle{\left(\left(\prod_{b=2}^n r_b\right)^K-r_1^K\right)^2} }
   { \displaystyle{\left(\prod_{b=2}^n \ol{\zeta_b}^{j_b}r_b - \ol{\zeta_1}^{j_1}r_1\right)^2}} \cdot \prod_{b=2}^n \frac{\left(1-r_b^{\ell_b}\right)^2}{ \left(1-\ol{\zeta_b}^{j_b}r_b\right)^2} \nonumber \right)\\
  &=  \frac{1}{\pi^n  L^2 \Delta}\,
   \sum_{j\in \mathfrak{B}}\left(
   \ol{\zeta_1}^{j_1}r_1^{1-\ell_1}\prod_{b=2}^n \ol{\zeta_b}^{2j_b}r_b^{2-\ell_b}\cdot
   \left( \sum_{\nu=0}^{K-1} \left(\prod_{b=2}^{n} \ol{\zeta_b}^{j_{b}}r_{b}\right)^{\nu} (\ol{\zeta_1}^{j_1} r_1)^{K-\nu -1}  \right)^{2} \times\right.\nonumber\\
   & \left.\phantom{\frac{1}{\pi^n \cdot L^2\cdot \Delta}\,
   \sum_{j\in \mathfrak{B}}
   \ol{\zeta_1}^{j_1}r_1^{1-\ell_1}\prod_{b=2}^n \ol{\zeta_b}^{2j_b}r_b^{2-\ell_b}} \times \prod_{b=2}^{n}\left( \sum_{m_{b}=0}^{\ell_b-1}
   (\ol{\zeta_b}^{j_{b}}r_{b})^{m_{b}} \right)^{2}\right) \label{eq-squared}\\
    &= \frac{1}{\pi^n  L^2 \Delta}\,\sum_{\alpha_1=0}^{2K-2}\sum_{\alpha_2=0}^{2K+2\ell_2-4}\dots \sum_{\alpha_n=0}^{2K+2\ell_n-4}A(\alpha)r_1^{\alpha_1+1-\ell_1} \prod_{b=2}^n r_b^{\alpha_b+2-\ell_b}\label{eq-lastbut}\\
  &= \frac{1}{\pi^n  L^2 \Delta}\,\sum_{\alpha_1=1-\ell_1}^{2K-\ell_1-1}\sum_{\alpha_2=2-\ell_2}^{2K+\ell_2-2}\dots \sum_{\alpha_n=2-\ell_n}^{2K+\ell_n-2}\widetilde{A}(\alpha)r^\alpha,\label{eq-lastline}
\end{align}
where in \eqref{eq-lastbut}, for simplicity of notation, we have expressed the 
quantity under the summation sign in \eqref{eq-squared} as a (Laurent) polynomial in the $n$ variables $(r_1, \dots, r_n)$ with coefficients $A(\alpha)\in \cx$. In \eqref{eq-lastline}, we have re-indexed the sum, and we denote
 $r^\alpha=r_1^{\alpha_1}\dots r_n^{\alpha_n}$. Also, $\widetilde{A}(\alpha) = A(\alpha_1+\ell_1-1,\alpha_2+\ell_2-2,\dots,\alpha_n+\ell_n-2)$. 

Notice that \eqref{eq-lastline} is a multi-variable polynomial in $(r_1,\dots, r_n)$.  Then, by the invariance  of $\widehat{B}$ shown in \eqref{eq-invariance}, we can replace   the variable $r_a$, with $1\leq a \leq n$ by $\ol{\zeta_a}r_a$, and the value of the polynomial remains unchanged 
\begin{align*}
&\frac{1}{\pi^n  L^2 \Delta}\,\sum_{\alpha_1=1-\ell_1}^{2K-\ell_1-1}\sum_{\alpha_2=2-\ell_2}^{2K+\ell_2-2}\dots \sum_{\alpha_n=2-\ell_n}^{2K+\ell_n-2}\widetilde{A}(\alpha)r^\alpha \\= &\frac{1}{\pi^n  L^2 \Delta}\,\sum_{\alpha_1=1-\ell_1}^{2K-\ell_1-1}\sum_{\alpha_2=2-\ell_2}^{2K+\ell_2-2}\dots \sum_{\alpha_n=2-\ell_n}^{2K+\ell_n-2}\ol{\zeta}^{\alpha_{a}}\widetilde{A}(\alpha)r^\alpha .\end{align*}
Looking at the difference of the two sides of the above equation, we see that for each $r=(r_1,\dots, r_n)$ and each $1\leq a\leq n$, we have
\[
  \sum_{\alpha_1=1-\ell_1}^{2K-\ell_1-1}\sum_{\alpha_2=2-\ell_2}^{2K+\ell_2-2}\dots \sum_{\alpha_n=2-\ell_n}^{2K+\ell_n-2}(\ol{\zeta}^{\alpha_{a}}-1)\widetilde{A}(\alpha)r^\alpha = 0.
\]
This is a polynomial in $r$ which vanishes identically, so 
each of its coefficients is zero.  This implies that for a fixed $\alpha$, the quantity $\widetilde{A}(\alpha)$ can be non-zero  only if $(\ol{\zeta}^{\alpha_a}-1) = 0$. 
Since this holds for each $1\leq a \leq n$, the only terms in 
 \eqref{eq-lastline} that survive are the ones in which the  monomial $r^\alpha=r_1^{\alpha_1}r_2^{\alpha_2}\dots r_n^{\alpha_n}$ is of the form
 \[\alpha= \ell\cdot\beta = :(\ell_1\beta_1, \ell_2\beta_2, \dots,\ell_n\beta_n), \]
 for some $\beta \in \mathbb{Z}^n$.
 From the bounds on the indices $\alpha_c$ in  \eqref{eq-lastline}, this implies that the indices corresponding to possibly nonzero terms are the 
 following multiples of $\ell_c$: 
 \begin{equation}\label{eq-alphaclimits} \alpha_{c} = 0,\, \ell_c,\, 
 \dots,\, 2K\quad \text{for each }\quad 2 \leq c \leq n \text{ if } \ell_c \neq 1, \end{equation} 
 and 
 \begin{equation}\label{eq-alphaclimits2} \alpha_{c} = 1,\, 
 \dots,\, 2K-1\quad \text{for each }\quad 2 \leq c \leq n  \text{ if } \ell_c=1 \end{equation} 
 since for these (and only these) $\alpha_c$, we have $2 -\ell_c\leq \alpha_c \leq 2K+\ell_c-2$, and $\alpha_c$ is divisible by $\ell_c$. Recall here that by Lemma~\ref{lem-numbertheoretic}, the integer $K=\mathrm{lcm}(k_1,\dots, k_n)$ is divisible by $\ell_{c}$, since we also have $K=\mathrm{lcm}(\ell_1,\dots,\ell_n)$.
 Similar arguments also show that the indices $\alpha_1$ for which we can have possibly nonzero terms in \eqref{eq-lastline} are
 \begin{equation}\label{eq-alpha1limits}
  \alpha_{1} = 0,\, \ell_1,\, \dots, \,2K - 2\ell_{1}.
  \end{equation}
 Using the representation $\alpha=\ell\cdot \beta=(\ell_1\beta_1, \dots, \ell_n\beta_n)$, we see that these same
 indices are also described by the collection $\mathfrak{G}^*(k)$ of $\beta\in \mathbb{Z}^n$ 
 such that
 \begin{equation}
     \label{eq-gstar1}0\leq \beta_1\leq\frac{2K}{\ell_1}-2=2k_1-2,
 \end{equation}
 and for each $2 \leq b \leq n$
 \begin{equation}
     \label{eq-gstar2}\begin{cases} 0\leq\beta_b\leq\dfrac{2K}{\ell_b}=2k_b &\text{ if } \ell_b\not=1, \\ 1 \leq \beta_b \leq 2K-1=2k_b-1&\text{ if } \ell_b=1 .\end{cases}
 \end{equation}
Notice that the set $\mathfrak{G}$ of \eqref{eq-G} is contained in $\mathfrak{G}^*(k)$.
We can now write
\begin{align}
    \widehat{B}(r_1,\dots, r_n)&= \eqref{eq-lastline}
    = \frac{1}{\pi^n \cdot L^2\cdot \Delta}\,\sum_{\beta\in \mathfrak{G}^*(k)}\widetilde{A}(\ell\cdot \beta)\,r^{\ell\cdot\beta},\label{eq-aftervanishing}
    \end{align}
which follows from combining equations  \eqref{eq-alphaclimits} through  \eqref{eq-alpha1limits}.

We now proceed to compute the coefficients $\wt{A}(\ell\cdot \beta)$. Introduce, a set of indices $\mathfrak{C}\subset\mathbb{Z}^{n-1}$ by setting
\begin{equation}
    \label{eq-gothicc}\mathfrak{C}=\{(m_2,\dots, m_n)\in \mathbb{Z}^{n-1}\mid
0\leq m_b \leq \ell_b-1 \quad \text{ for } 2 \leq b \leq n\}.
\end{equation}
Now, in \eqref{eq-squared}, we rewrite the first square factor as a product of two sums over indices $\nu$ and $N$:
\begin{align*}
   &\left( \sum_{\nu=0}^{K-1} \left(\prod_{b=2}^{n} \ol{\zeta_b}^{j_{b}}r_{b}\right)^{\nu} (\ol{\zeta_1}^{j_1} r_1)^{K-\nu -1}  \right)^{2} =\\
   &\left( \sum_{\nu=0}^{K-1} \left(\prod_{b=2}^{n} \ol{\zeta_b}^{j_{b}}r_{b}\right)^{\nu} (\ol{\zeta_1}^{j_1} r_1)^{K-\nu -1}  \right)
   \left( \sum_{N=0}^{K-1} \left(\prod_{b=2}^{n} \ol{\zeta_b}^{j_{b}}r_{b}\right)^{N} (\ol{\zeta_1}^{j_1} r_1)^{K-N -1}  \right).
\end{align*}
Similarly, writing each of the other $(n-2)$ square  factors $\left( \sum_{m_{b}=0}^{\ell_b-1}
   (\ol{\zeta_b}^{j_{b}}r_{b})^{m_{b}} \right)^{2}$  for $2\leq b \leq n$
in \eqref{eq-squared} as a product of sums over different indices $m_b$ and $M_b$ and then expanding the products we can rewrite \eqref{eq-squared} as
\begin{align} &\widehat{B}(r_1,\dots, r_n)\nonumber\\
=& \frac{1}{\pi^n L^2\Delta } \sum_{j \in \mathfrak{B}}\,\sum_{m, M \in \mathfrak{C}}\, \sum_{\nu,N = 0}^{K-1}\ol{\zeta_1}^{j_1(2K-\nu-N-1)}r_1^{2K-\nu-N-\ell_1-1} \prod_{b=2}^n \ol{\zeta}_b^{j_b(\nu+N+m_b+M_b+2)}r_b^{m_b+M_b+\nu+N-\ell_b+2} ,\label{eq-khk2} \end{align}
where, in the sum above, $j=(j_1,\dots, j_n)$ ranges over the set $\mathfrak{B}$ of \eqref{eq-gothicb}, and $m=(m_2,\dots, m_n)$ and $M=(M_2,\dots, M_n)$ are multi-indices that range over the set $\mathfrak{C}$
of \eqref{eq-gothicc}, and the indices $\nu$ and $N$ each go independently from $0$ to $K-1$.
To find $\wt{A}(\ell\cdot \beta)$, note that in the sum  \eqref{eq-khk2}, we are considering 
those terms in which the power of $r_1$ is $\ell_1\beta_1$ and the power of $r_b$ is $\ell_b\beta_b$ for $2\leq b \leq n$. Notice that for these powers of $r_j$, the powers of $\zeta_j$'s are each 1.
Therefore, comparing the two expressions 
\eqref{eq-khk2} and \eqref{eq-aftervanishing} for $\widehat{B}(r_1,\dots, r_n)$, we conclude that  for each $\beta\in \mathfrak{G}^*(k)$ we have
\begin{align}
    \wt{A}(\ell\cdot\beta)& ={\sideset{}{'}\sum} \ol{\zeta_1}^{j_1(2K-\nu-N-1)} \prod_{b=2}^n \ol{\zeta}_b^{j_b(\nu+N+m_b+M_b+2)} \label{eq-atilde-1} \\
    &= {\sideset{}{'}\sum} 1,\label{eq-atilde-2}
\end{align}
where ${\sideset{}{'}\sum}$ denotes a sum extending over the set of indices 
$j=(j_1,j_2,\dots, j_n), m=(m_2,\dots, m_n), M=(M_2,\dots, M_n)$ 
and $\nu, N$ ranging over
\[ \begin{cases} j\in \mathfrak{B},\, m, M \in \mathfrak{C}\\
0\leq \nu, N \leq K-1\\
 m_b + M_b + \nu + N + 2 - \ell_b = \beta_b \ell_b, \quad \text{for each}\quad 2\leq b \leq n\\
 2K-\nu-N-\ell_1-1 = \beta_1\ell_1.
 \end{cases} \]
 The expression in \eqref{eq-atilde-2} follows from \eqref{eq-atilde-1} since for each such index, the summand is clearly 1. Observe now that 
 in the range of summation described above, the indices $j=(j_1,j_2,\dots, j_n)\in \mathfrak{B}$ (with $\mathfrak{B}$ as in \eqref{eq-gothicb}) vary freely without any interaction with the other indices $m, M, \nu, N$.
 Therefore,
\begin{equation}\label{eq-hkbeta}
    \wt{A}(\ell\cdot\beta)=\eqref{eq-atilde-2} = \sum_{j\in \mathfrak{B}} C(\beta) = \abs{\mathfrak{B}}\cdot C(\beta)=L\cdot C(\beta),
\end{equation}
where as in the statement of the theorem, $L= \prod_{a=1}^n \ell_a$, and  $C(\beta)$ is the number of solutions in integers $m=(m_2,\dots,m_n), \, M=(M_2,\dots, M_n),\, \nu,\, N$ of the system of equations and inequalities given by
\begin{align}
\begin{cases}
    0 \leq m_b, M_b \leq \ell_b-1, &\text{ for each } 2\leq b \leq n \nonumber\\ 
    0 \leq \nu,N \leq K-1,\\ 
    m_b + M_b + \nu + N = \ell_b(\beta_b + 1) -2 & \text{ for each } 2\leq b \leq n.  \\
    \nu + N = 2K - \ell_1(\beta_1 + 1)-1. & 
    \end{cases}
\end{align}
To find $C(\beta)$, we first note that the third equation may be replaced (with the help of the last equation) by the equivalent equation
\begin{equation}\label{condition A'}
m_b+M_b=\ell_b(\beta_b+1)+\ell_1(\beta_1+1)-2K-1 \text{ for each } 2\leq  b \leq n.
\end{equation}
Consequently, the number of solutions $C(\beta)$ of
the system  can be obtained by  
 multiplying together the number of solutions of
 \[ \nu + N = 2K - \ell_1(\beta_1 + 1)-1, \quad  0 \leq \nu,N \leq K-1\] with the number of solutions for each $b,$ with $ 2\leq b \leq n$ to 
 \[ m_b+M_b=\ell_b(\beta_b+1)+\ell_1(\beta_1+1)-2K-1,\quad  0 \leq m_b, M_b \leq \ell_b-1. \]
 To represent these numbers, 
for integers $\lambda, \mu$, define $\mathsf{D}_\lambda(\mu)$ to be the number of integer solutions $(x,y)\in \mathbb{Z}^2$  of the system of equations and inequalities:
 \begin{align}
     &x+y=\mu, \label{eq-xplusy}\\
     &0\leq x \leq \lambda-1,\label{eq-xbounds}  \\
    & 0 \leq y \leq \lambda-1. \label{eq-ybounds}
 \end{align}
 Then clearly we have
 \begin{equation} \label{eq-cgeneral}
   C(\beta)= \mathsf{D}_K(2K-\ell_{1}(\beta_{1}+1)-1)\cdot\prod_{b=2}^n \mathsf{D}_{\ell_b} \left( \ell_{b}(\beta_b +1) +\ell_{1}(\beta_{1} + 1) -2K-1 \right) .   
 \end{equation}
 
 \textbf{Claim:} the numbers $\mathsf{D}_\lambda(\mu)$ are given by the formula
 \eqref{eq-dcombinations} that precedes the statement of Theorem~\ref{thm-computation}.

  Indeed, if $\mu \leq -1$, then by \eqref{eq-xplusy}, we have  $x+y \leq -1$. However, from \eqref{eq-xbounds} and \eqref{eq-ybounds} in the definition of $\mathsf{D}_\lambda(\mu)$, this is impossible. Hence, $\mathsf{D}_\lambda(\mu) = 0$. Similarly, if  $\mu \geq 2\lambda-1$, then by \eqref{eq-xplusy},  $x+y \geq 2\lambda-1$. However, from \eqref{eq-xbounds} and \eqref{eq-ybounds} in the definition of $\mathsf{D}_\lambda(\mu)$, this is impossible. Hence, $\mathsf{D}_\lambda(\mu) = 0$.

In the other cases, it is easy to enumerate the solutions. If $0 \leq \mu \leq \lambda-1$, then
\[ \mathsf{D}_\lambda(\mu)=  \abs{ \{(x, \mu-x) 	:  0 \leq x \leq \mu\} } = \mu +1,\]
and if $\lambda \leq \mu \leq 2\lambda-2$, then 
\[\mathsf{D}_\lambda(\mu) = \abs{\{(x, \mu-x) :  \mu -\lambda + 1 \leq x \leq \lambda-1\}} = 2\lambda-1-\mu,\]
completing the proof of the claim.

From \eqref{eq-aftervanishing} and \eqref{eq-hkbeta} we see that
\begin{equation}
    \label{eq-penultimate}
  \mathbb{B}_{\mathscr{H}(k)}(\phi(z),w)=\widehat{B}(r_1,\dots, r_n)=\frac{1}{\pi^n L^2\Delta}\sum_{\beta\in \mathfrak{G}^*(k)}L\cdot C(\beta) r^{\ell\cdot\beta}= \frac{1}{\pi^nL\Delta}\sum_{\beta\in \mathfrak{G}^*(k)} C(\beta)\, r^{\ell\cdot\beta}.  
\end{equation}
Now
\[ \phi(z)=(\phi_1(z),\dots, \phi_n(z))=(z_1^{\ell_1},\dots, z_n^{\ell_n}).\]
Therefore, recalling the definition \eqref{eq-ra}, we see that
\begin{align*}
     r^{\ell\cdot \beta}&=(r_1^{\ell_1})^{\beta_1}\cdots (r_n^{\ell_n})^{\beta_n}\\&=(z_1^{\ell_1}\ol{w_1})^{\beta_1}\dots (z_n^{\ell_n}\ol{w_n})^{\beta_n}\\
     &= (\phi_1(z)\ol{w_1})^{\beta_1}\cdots(\phi_n(z)\ol{w_n})^{\beta_n} .
\end{align*}
Also, remembering that $\ell_b=\dfrac{K}{k_b}$ for each $b$, we have
\[ r_b^K=(z_b ^{\ell_b})^{k_b} \ol{w_b}^{k_b}=\phi_b(z)^{k_b}\ol{w_b}^{k_b}, \]
and
\[r_b^{\ell_b}=z_b^{\ell_b}\ol{w_b}=\phi_b(z)\ol{w_b}. \]
Therefore, recalling the definition \eqref{eq-Delta}, we have
\begin{align*}
    \Delta &= \left(\left(\prod_{b=2}^n r_b\right)^K-r_1^K\right)^2\cdot \prod_{b=2}^n\left(1-r_b^{\ell_b}\right)^2\\
    &= \left(\left(\prod_{b=2}^n \phi_b(z)^{k_b}\ol{w_b}^{k_b}\right)-\phi_1(z)^{k_1}\ol{w_1}^{k_1}\right)^2\cdot \prod_{b=2}^n\left(1-\phi_b(z)\ol{w_b}\right)^2.
\end{align*}
Therefore, if we replace $\phi(z)$ by $z$ in the first member of \eqref{eq-penultimate}, we see that the last member is transformed
to a function of $(t_1,\dots, t_n)$, where $t_a=z_a\ol{w_a}.$ In fact, we get \eqref{eq-kscrh}, thus completing the proof of  the result, except that in the numerator of \eqref{eq-kscrh} we have obtained the polynomial $\sum_{\beta\in \mathfrak{G}^*(k)} C(\beta)t^\beta$ instead of $\sum_{\beta\in \mathfrak{G}} C(\beta)t^\beta$. Therefore, to complete the proof, we need to show that 
if $\beta\in \mathfrak{G}^*(k)\setminus \mathfrak{G}$ then $C(\beta)=0.$ Now for such a $\beta$, there exists a $2\leq b\leq n$ such that $\ell_b=1$ and $\beta_b$ is either 0 or $2k_b$. First assume that
$\beta_b=0$. Then the factor $\mathsf{D}_{\ell_b}(\ell_b(\beta_b+1)+\ell_1(\beta_1+1)-2K-1)$ 
in the formula
\eqref{eq-cbeta} reduces to $\mathsf{D}_1(\ell_1(\beta_1+1)-2K)$. By the definition \eqref{eq-dcombinations} of $\mathsf{D}$, 
this is not zero if and only if $\ell_1(\beta_1+1)-2K=0$. However, in the latter case, 
we have the first factor of \eqref{eq-cbeta} equal to zero, since it equals
$\mathsf{D}_K(-1)$. 

In the other case $\beta_b=2k_b=2K$ we see that the factor $\mathsf{D}_{\ell_b}(\ell_b(\beta_b+1)+\ell_1(\beta_1+1)-2K-1)$  reduces to $\mathsf{D}_1(\ell_1(\beta_1+1))=0$. 
\end{proof}
\subsection{ Recapturing the special cases $\mathscr{H}(1, -k)$ and  $\mathscr{H}(k, -1)$}
\label{sec-luke}
We now show that the results of \cite{Edh16} on explicit Bergman kernels of fat and thin Hartogs triangles are special cases of Theorem~\ref{thm-computation}.

\subsubsection{$\mathscr{H}(1, -k), k \geq 1$}

 We follow the notation used in Theorem~\ref{thm-computation}. For $\mathscr{H}(1, -k)$ we have $k_1 = 1$ and $k_2 = k$. Hence $K= \mathrm{lcm}(1, k) = k$ and $L = k$. We then have \[ \mathfrak{G} = \{ (\beta_1, \beta_2) \in \mathbb Z^2 \,\,|\,\, \beta_1 = 0, \,\, 0 \leq \beta_2 \leq 2k \}.\]

 For $(0,\beta_2)\in \mathfrak{G}$, we compute $C(0,\beta_2)$,  where $0 \leq \beta_2 \leq 2k$. By \eqref{eq-cbeta}, we have
 \[C(0,\beta_2) = \mathsf{D}_k(k-1) \mathsf{D}_1(\beta_2 - k).\] Now from \eqref{eq-dcombinations}, we have  $\mathsf{D}_k(k-1) = k$ and \begin{equation*}
       \mathsf{D}_1(\beta_2-k) = 
       \begin{cases} 
      0 & 0 \leq \beta_2 \leq k-1\\
   1  & \beta_2 = k\\
       0 & k+1 \leq \beta_2 \leq 2k. \\
       \end{cases}
   \end{equation*}
Hence for $\beta = (0, \beta_2) \in \mathfrak{G}$, $C(\beta) \neq 0$ if and only if $\beta_2 = k$ and in this case, $C(\beta) = k.$ 
Hence the formula (\ref{eq-kscrh}) gives 

\[ \mathbb{B}_{\mathscr{H}(1, -k)}(z,w) =  \frac{1}{\pi^2 k}
\cdot \frac{ \displaystyle{k\,\, t_2^k}}{\displaystyle{(t_2^k-t_1)^2 (1-t_2)^2}} =  \frac{1}{\pi^2 }
\cdot \frac{ \displaystyle{ t_2^k}}{\displaystyle{(t_2^k-t_1)^2 (1-t_2)^2}} ,\]
which precisely is the content of  \cite[Theorem 1.4]{Edh16}. 

\subsubsection{$\mathscr{H}(k, -1), k \geq 2$}

In this case, $k_1 = k$ and $k_2 = 1$. Hence $K= k$ and $L= k$. We then have \[ \mathfrak{G} = \{ (\beta_1, \beta_2) \in \mathbb Z^2 \,\,|\,\, 0 \leq \beta_1 \leq 2k-2, \,\, 0 \leq \beta_2 \leq 2 \},\]

and

\begin{equation*}
       C(\beta) = 
       \begin{cases} 
      \mathsf{D}_k(2k-\beta_1-2) \mathsf{D}_k(\beta_1-k),  &  \beta = (\beta_1, 0)\\
    \mathsf{D}_k(2k-\beta_1-2) \mathsf{D}_k(\beta_1),  &  \beta = (\beta_1, 1)\\
     \mathsf{D}_k(2k-\beta_1-2) \mathsf{D}_k(\beta_1 + k),   &  \beta = (\beta_1, 2). \\
       \end{cases}
   \end{equation*}

We compute $\mathsf{D}_k$'s. 
\begin{align*}
\mathsf{D}_k(2k-\beta_1-2) &= 
       \begin{cases} 
       \beta_1 + 1,  &  \,\,0 \leq \beta_1 \leq k-1\\
       2k-\beta_1-1, & \,\, k \leq \beta_1 \leq 2k-2.\\
     \end{cases}\\
       \mathsf{D}_k(\beta_1-k)& = 
       \begin{cases} 
        0, &   0 \leq \beta_1 \leq k-1 \\
      \beta_1-k+1,  &  k \leq \beta_1  \leq 2k-2.\\
  \end{cases}\\
      \mathsf{D}_k(\beta_1)& = 
       \begin{cases} 
        \beta_1 + 1,  &   0 \leq \beta_1 \leq k-1 \\
      2k-1-\beta_1, &  k \leq \beta_1  \leq 2k-2.\\
  \end{cases}\\
       \mathsf{D}_k(\beta_1+k)& = 
       \begin{cases} 
        k-\beta_1 -1,   &   0 \leq \beta_1 \leq k-1 \\
    0, &  k \leq \beta_1  \leq 2k-2.\\
  \end{cases}\\
   \end{align*}
   
  Hence, \[\displaystyle{\sum_{\beta\in \mathfrak{G}} C(\beta) t^\beta} = \underbrace{\sum_{\beta_1 = k}^{2k-2} (2k-\beta_1-1)(\beta_1-k+1) t_1^{\beta_1}}_{\beta_2=0} +  \underbrace{\left(\sum_{\beta_1 = 0}^{k-1}(\beta_1+1)^2 t_1^{\beta_1} t_2 + \sum_{\beta_1 = k}^{2k-2} (2k-\beta_1-1)^2 t_1^{\beta_1} t_2 \right)}_{\beta_2=1}  \]
  
  \[ + \underbrace{\sum_{\beta_1 = 0}^{k-1} (\beta_1+1)(k-\beta_1-1)  t_1^{\beta_1} t_2^2}_{\beta_2=2}. \]
  
 We rewrite the terms corresponding to $\beta_2= 0, 1, 2$ as follows.
 In the term for $\beta_2=0$, by making the substitution $\ell =\beta_1 - k +1$, we obtain $\displaystyle{\left(\sum_{\ell = 1}^{k-\ell} (k-\ell) \ell\cdot  t_1^{\ell-1}\right)t_1^k}.$

 In the first sum of the second term (which corresponds to $\beta_2=1$) we make the substitution $\ell=\beta_1+1$,  which transforms it into $\displaystyle{\sum_{ \ell= 1}^{k} \ell^2 \cdot t_1^{\ell-1} t_2}$. In the second sum, we make the substitution $\ell = \beta_1 -k + 1$, which transforms it into  $\displaystyle{\sum_{\ell=1}^{k} (k-\ell)^2\cdot t_1^{k+\ell-1} t_2}$. Combining the two we can represent the second term as

 \[\sum_{ \ell= 1}^{k} \ell^2 t_1^{\ell-1} t_2 + \sum_{\ell=1}^{k} (k-\ell)^2 t_1^{k+\ell-1} t_2 =  \left(\sum_{ \ell= 1}^{k} (\ell^2 +(k-\ell)^2) t_1^k)  t_1^{\ell-1}\right) t_2.\]
 
 Similarly using the substitution $\ell = \beta_1 + 1$, the last term becomes
 \[\sum_{\beta_\ell = 0}^{k-1} (\beta_1+1)(k-\beta_1-1)  t_1^{\beta_1} t_2^2 = \sum_{\ell= 1}^{k} \ell \,\,(k- \ell)  t_1^{\ell-1} t_2^2. \]
 
 Therefore we get the expression for the Bergman kernel for $\mathscr{H}(k, -1)$ as
 \[
 \frac{1}{\pi^2 k}
\cdot \frac{\displaystyle{\left(\sum_{\ell = 1}^{k-1} (k-\ell) \cdot \ell \cdot t_1^{\ell-1}\right)t_1^k +  \left(\sum_{ \ell= 1}^{k} (\ell^2 +(k-\ell)^2) t_1^k)  t_1^{\ell-1}\right) t_2 + \left(\sum_{\ell= 1}^{k} \ell \,\,(k- \ell)  t_1^{\ell-1}\right) t_2^2}}{\displaystyle{(t_2-t_1^k)^2 (1-t_2)^2}} .\]

The above expression is precisely the statement of \cite[Theorem 1.2]{Edh16}.
%%%%%%%%%%%%%%%%%%%%%%%%%%%%%%%%%%%%%%%
\bibliographystyle{alpha}
\bibliography{bergman}
\end{document}